\documentclass[a4paper,11pt]{amsart}
\pdfoutput=1
\usepackage{amssymb,amsmath,amsthm,amsfonts,centernot}
\usepackage[colorlinks=true,linkcolor=blue,citecolor=blue,urlcolor=cyan, pdfpagelabels=false]{hyperref}
\usepackage{amsmath,tikz}
\usepackage{dsfont}
\usetikzlibrary{matrix}
\usepackage{perpage} 
\MakePerPage{footnote}
\begin{document}
\numberwithin{equation}{section}
\newtheorem{theorem}{Theorem}
\newtheorem{algo}{Algorithm}
\newtheorem{lem}{Lemma} 
\newtheorem{de}{Definition} 
\newtheorem{ex}{Example}
\newtheorem{pr}{Proposition} 
\newtheorem{claim}{Claim} 
\newtheorem{re}{Remark}
\newtheorem{co}{Corollary}
\newtheorem{conv}{Convention}
\newcommand{\di}{\hspace{1.5pt} \big|\hspace{1.5pt}}
\newcommand{\idi}{\hspace{.5pt}|\hspace{.5pt}}
\newcommand{\hs}{\hspace{1.3pt}}
\newcommand{\thmf}{Theorem~1.15$'$}
\newcommand{\ndi}{\centernot{\big|}}
\newcommand{\nidi}{\hspace{.5pt}\centernot{|}\hspace{.5pt}}
\newcommand{\lp}{\mbox{$\hspace{0.12em}\shortmid\hspace{-0.62em}\alpha$}}
\newcommand{\PQ}{\bb{P}^1(\bb{Q})}
\newcommand{\pmn}{\cl{P}_{M,N}}
\newcommand{\lcm}{\operatorname{lcm}}
\newcommand{\he}{holomorphic eta quotient\hspace*{2.5pt}}
\newcommand{\hes}{holomorphic eta quotients\hspace*{2.5pt}}
\newcommand{\defG}{Let $G\subset\GG$ be a subgroup that is conjugate to a finite index subgroup of $\G$. } 
\newcommand{\defg}{Let $G\subset\GG$ be a subgroup that is conjugate to a finite index subgroup of $\G$\hs\hs} 
\renewcommand{\phi}{\varphi}
\newcommand{\Z}{\bb{Z}}
\newcommand{\ZD}{\Z^{\D}}
\newcommand{\N}{\bb{N}}
\newcommand{\Q}{\bb{Q}}
\newcommand{\pii}{{{\pi}}}
\newcommand{\R}{\bb{R}}
\newcommand{\C}{\bb{C}}
\newcommand{\I}{\hs\cl{I}_{n,N}}
\newcommand{\St}{\operatorname{Stab}}
\newcommand{\D}{\cl{D}_N}
\newcommand{\rh}{{{\boldsymbol\rho}}}
\newcommand{\bh}{{\cl{M}}} 
\newcommand{\lv}{\hyperlink{level}{{\text{level}}}\hspace*{2.5pt}}
\newcommand{\fct}{\hyperlink{factor}{{\text{factor}}}\hspace*{2.5pt}}
\newcommand{\q}{\hyperlink{q}{{\mathbin{q}}}}
\newcommand{\rd}{\hyperlink{redu}{{{\text{reducible}}}}\hspace*{2.5pt}}
\newcommand{\ird}{\hyperlink{irredu}{{{\text{irreducible}}}}\hspace*{2.5pt}}
\newcommand{\str}{\hyperlink{strong}{{{\text{strongly reducible}}}}\hspace*{2.5pt}}
\newcommand{\rdn}{\hyperlink{redon}{{{\text{reducible on}}}}\hspace*{2.5pt}}
\newcommand{\atl}{\hyperlink{atinv}{{\text{Atkin-Lehner involution}}}\hspace*{3.5pt}}
\newcommand{\atls}{\hyperlink{atinv}{{\text{Atkin-Lehner involutions}}}\hspace*{3.5pt}}
\newcommand{\T}{\mathrm{T}}
\renewcommand{\H}{\fr{H}}
\newcommand{\W}{\text{\calligra W}_n}
\newcommand{\GG}{\cl{G}}
\newcommand{\g}{\fr{g}}
\newcommand{\Gm}{\Gamma}
\newcommand{\Gmtl}{\widetilde{\Gamma}_\ell}
\newcommand{\gm}{\gamma}
\newcommand{\go}{\gamma_1}
\newcommand{\gmt}{\widetilde{\gamma}}
\newcommand{\gmdt}{\widetilde{\gamma}'}
\newcommand{\gmot}{\widetilde{\gamma}_1}
\newcommand{\gmdot}{{\widetilde{\gamma}}'_1}
\newcommand{\s}{\Large\text{{\calligra r}}\hspace{1.5pt}}
\newcommand{\ms}{m_{{{S}}}}
\newcommand{\nisim}{\centernot{\sim}}
\newcommand{\level}{\hyperlink{level}{{\text{level}}}}
\newcommand{\Redcon}{the \hyperlink{red}{\text{Reducibility~Conjecture}}}
\newcommand{\ReDcon}{The \hyperlink{red}{\text{Reducibility~Conjecture}}}
\newcommand{\Conred}{Conjecture~$1'$}
\newcommand{\effth}{Theorem~$\ref{27.11.2015}'$}
\newcommand{\Conredd}{Conjecture~$1''$}
\newcommand{\Conreddd}{Conjecture~$1'''$}
\newcommand{\Conired}{Conjecture~$2'$}
\newtheorem*{pro}{\textnormal{\textit{Proof of Lemma~\ref{27.11.2015.1}}}}
\newtheorem*{cau}{Caution}
\newtheorem{thrmm}{Theorem}[section]
\newtheorem*{thmA}{Theorem~A}
\newtheorem*{corA}{Corollary~A}
\newtheorem*{corB}{Corollary~B}
\newtheorem*{corC}{Corollary~C}
\newtheorem{no}{Notation}
\renewcommand{\thefootnote}{\fnsymbol{footnote}}
\newtheorem{oq}{Open question}
\newtheorem{conj}{Conjecture}
\newtheorem{hy}{Hypothesis} 
\newtheorem{expl}{Example}
\newcommand\ileg[2]{\bigl(\frac{#1}{#2}\bigr)}
\newcommand\leg[2]{\Bigl(\frac{#1}{#2}\Bigr)}
\newcommand{\e}{\eta}
\newcommand{\sgn}{\operatorname{sgn}}
\newcommand{\bb}{\mathbb}
\newtheorem*{conred}{Conjecture~\ref{con1}$\mathbf{'}$}
\newtheorem*{conredd}{Conjecture~\ref{con1}$\mathbf{''}$}
\newtheorem*{conreddd}{Conjecture~\ref{con1}$\mathbf{'''}$}
\newtheorem*{conired}{Conjecture~\ref{19.1Aug}$\mathbf{'}$}
\newtheorem*{efth}{Theorem~\ref{27.11.2015}$\mathbf{'}$}
\newtheorem*{eflem}{Theorem~\ref{15May15}.$(b)\mathbf{'}$}
\newtheorem*{procl}{\textnormal{\textit{Proof}}}
\newtheorem*{thmff}{Theorem~\ref{7.4Jul}$\mathbf{'}$}
\newtheorem*{coo}{Corollary~\ref{17Aug}$\mathbf{'}$}
\newcommand{\cooo}{Corollary~\ref{17Aug}$'$}
\newtheorem*{cotw}{Corollary~\ref{17.1Aug}$\mathbf{'}$}
\newcommand{\cotww}{Corollary~\ref{17.1Aug}$'$}
\newtheorem*{cothr}{Corollary~\ref{17.2Aug}$\mathbf{'}$}
\newcommand{\cothre}{Corollary~\ref{17.2Aug}$'$}
\newtheorem*{cne}{Corollary~\ref{15.5Aug}$\mathbf{'}$}
\newcommand{\cnew}{Corollary~\ref{15.5Aug}$'$\hspace{3.5pt}}
\newcommand{\fr}{\mathfrak}
\newcommand{\cl}{\mathcal}
\newcommand{\rad}{\mathrm{rad}}
\newcommand{\ord}{\operatorname{ord}}
\newcommand{\m}{\setminus}
\newcommand{\G}{\Gamma_1}
\newcommand{\GN}{\Gamma_0(N)}
\newcommand{\X}{\widetilde{X}}
\renewcommand{\P}{{\textup{p}}} 
\newcommand{\al}{{\hs\operatorname{al}}}
\newcommand{\p}{p_\text{\tiny (\textit{N})}}
\newcommand{\pN}{p_\text{\tiny\textit{N}}}
\newcommand{\bt}{\mbox{$\raisebox{-0.59ex}
  {${{l}}$}\hspace{-0.215em}\beta\hspace{-0.88em}\raisebox{-0.98ex}{\scalebox{2}
  {$\color{white}.$}}\hspace{-0.416em}\raisebox{+0.88ex}
  {$\color{white}.$}\hspace{0.46em}$}{}}
  \newcommand{\un}{\hs\underline{\hspace{5pt}}\hs}
\newcommand{\U}{u_\textit{\tiny N}}
\newcommand{\Upr}{u_{\textit{\tiny N}^\prime}}
\newcommand{\Up}{u_{\textit{\tiny p}^\textit{\tiny e}}}
\newcommand{\Un}{u_{\textit{\tiny p}_\textit{\tiny 1}^{\textit{\tiny e}_\textit{\tiny 1}}}}
\newcommand{\Um}{u_{\textit{\tiny p}_\textit{\tiny m}^{\textit{\tiny e}_\textit{\tiny m}}}}
\newcommand{\Ut}{u_{\text{\tiny 2}^\textit{\tiny a}}}
\newcommand{\At}{A_{\text{\tiny 2}^\textit{\tiny a}}}
\newcommand{\Uh}{u_{\text{\tiny 3}^\textit{\tiny b}}}
\newcommand{\Ah}{A_{\text{\tiny 3}^\textit{\tiny b}}}
\newcommand{\Uprl}{u_{\textit{\tiny N}_1}}
\newcommand{\Uprlm}{u_{\textit{\tiny N}_i}}
\newcommand{\UM}{u_\textit{\tiny M}}
\newcommand{\UMp}{u_{\textit{\tiny M}_1}}
\newcommand{\w}{\omega_\textit{\tiny N}}
\newcommand{\wm}{\omega_\textit{\tiny M}}
\newcommand{\wa}{\omega_{\text{\tiny N}_\textit{\tiny a}}}
\newcommand{\wma}{\omega_{\text{\tiny M}_\textit{\tiny a}}}
\renewcommand{\P}{{\textup{p}}}

\title[ 
{\tiny Infinite families of simple holomorphic eta quotients}] 
{Infinite families of simple \\ holomorphic eta quotients} 

\author{Soumya Bhattacharya}
\address 
{Ramakrishna Mission Vivekananda University\\
Belur Math\\
Howrah 711202} 
\email{soumya.bhattacharya@gmail.com}

%


\maketitle

 \begin{abstract}
 We address the problem of constructing a simple holomorphic eta quotient of a given level $N$.
Such constructions are known for all cubefree $N$.
Here, we provide such constructions for
arbitrarily large prime power levels. 
As a consequence, we 
obtain an irreducibility
criterion for holomorphic eta quotients in general. 
 \end{abstract}

 \section{Introduction}
\label{intro}
The Dedekind eta function is defined by the infinite product:\hypertarget{queue}{}
 \begin{equation} \eta(z):=q^{\frac{1}{24}}\prod_{n=1}^\infty(1-q^n)\hspace{5pt}\text{for all $z\in\H$,}
\label{17.4Aug}\end{equation} 
where $q^r=q^r(z):=e^{2\pi irz}$ for all $r$ 
and 
$\H:=\{\tau\in\C\hs\idi\operatorname{Im}(\tau)>0\}$. 
Eta is a holomorphic function on $\H$ with no zeros.
This function has its significance in Number Theory. For example, 
$1/\e$ is the generating function for the ordinary partition function $p:\N\rightarrow\N$
(see \cite{ak})
and 
 the constant term in the Laurent
 expansion at 
 $1$ of 
the Epstein zeta function $\zeta_Q$ 
attached to a positive definite quadratic form $Q$ is related
 via the Kronecker limit formula to
the value of $\e$
at the root of the associated quadratic polynomial in $\H$ (see \cite{coh}).
The value of $\e$ at such a quadratic irrationality of discriminant $-D$
is also related via the Lerch/Chowla-Selberg formula to  the values of the Gamma function 
with arguments in $D^{-1}\N$ 
 \hspace{.5pt}(see 
  \cite{pw}). 
   Further, eta quotients 
  appear in 
  denominator formula for Kac-Moody algebras,
   (see 
   \cite{k}), 
 in 
 ''Moonshine`` of finite groups (see \cite{h}),
in Probability Theory, e.~g. in the distribution of the distance travelled in a uniform four-step
random walk (see \cite{bswz})
and in the distribution of crossing probability in two-dimensional percolation (see \cite{kz}).
In fact, the eta function
comes up naturally in many other areas of Mathematics (see the Introduction in \cite{B-three} for a brief
overview of them). 

The function $\e$ is a modular form
of weight $1/2$ with a multiplier system 
on $\operatorname{SL}_2(\Z)$ (see \cite{b}).
An 
eta quotient $f$ is a finite product of the form 
\begin{equation}
 \prod\e_d^{X_d},
\label{13.04.2015}\end{equation}
where $d\in\N$, $\eta_d$ is the rescaling of $\eta$ by $d$, defined by
\begin{equation}
 \e_d(z):=\e(dz) \ \text{ for all $z\in\H$}
\end{equation}
and $X_d\in\Z$.
Eta quotients naturally inherit modularity 
from $\e$: The eta quotient $f$ in (\ref{13.04.2015}) transforms like a modular form of
weight $\frac12\sum_dX_d$ with a multiplier system on suitable congruence subgroups of $\operatorname{SL}_2(\Z)$: 
The largest among
these subgroups is 
\begin{equation}
 \Gm_0(N):=\Big{\{}\begin{pmatrix}a&b\\ c&d\end{pmatrix}\in
\operatorname{SL}_2(\Z)\hspace{3pt}\Big{|}\hspace{3pt} c\equiv0\hspace*{-.3cm}\pmod N\Big{\}},
\end{equation}
where 
\begin{equation}
 N:=\operatorname{lcm}\{d\in\N\hs\idi\hs X_d\neq0\}.
\end{equation}
We call $N$ the \emph{level} of $f$.
Since $\eta$ is non-zero on $\H$, 
the eta quotient $f$ 
is holomorphic if and only if $f$ does not have any pole at the cusps of 
$\Gamma_0(N)$.

An \emph{eta quotient 
on $\Gm_0(M)$} is
an eta quotient whose level divides $M$.
Let $f$, $g$ and $h$ be nonconstant \hes on $\Gm_0(M)$
such that 
$f=g\times h$. Then we say that $f$ is \emph{factorizable on} $\Gm_0(M)$. 
We call a holomorphic eta quotient $f$ of level $N$ \emph{quasi-irreducible} (resp. \emph{irreducible}),
if it is not factorizable on $\Gm_0(N)$ (resp. on $\Gm_0(M)$ for all multiples $M$ of $N$).
Here, it is worth mentioning that the notions of irreducibility and quasi-irreducibility of holomorphic 
eta quotients are conjecturally equivalent (see \cite{B-three}).
We say that a holomorphic eta quotient is \emph{simple} if it is 
 both primitive and quasi-irreducible.
 
\section{The main result and conjecture}

\ \quad For a prime $p$ and an integer $n>3$, we define the eta quotient $f_{p,n}$ by 
 \begin{equation}
f_{p,n}:=\left\{\begin{array}{cl}
{\frac{\displaystyle{\eta_p^p\hspace{1pt}\eta_{p^{n-1}}^{(p-1)^2}\hspace{1pt}\prod_{s=1}^{n/2-1}\eta_{p^{2s-1}}^{p^2-3p+1}\hspace{1pt}\eta_{p^{2s}}^{p^2-2p+2}}}{\displaystyle{(\eta\hspace{1pt}\eta_{p^n})^{p-1}}}}&\text{ if $n$ is even.}\\     
\\
{\frac{\displaystyle{(\eta_p\hspace{1pt}\eta_{p^{n-1}})^p\hspace{1pt}\prod_{s=1}^{n-1}\eta_{p^s}^{p^2 - 3p + 2}}}{\ \quad\displaystyle{(\eta\hspace{1pt}\eta_{p^n})^{p-1}}}}&\text{ if $n$ is odd and $p\neq2$,}\\ 
                 \end{array}\right.\label{fpn}\end{equation}
Clearly, 
$f_{p,n}$ is invariant under the Fricke involution $W_{p^n}$.
We shall show that: 
 \begin{theorem}
 For any integer $n>3$, 
 $f_{p,n}$ is a simple holomorphic eta quotient of level $p^n$.
\label{prp1}\end{theorem}

From Theorem~2 in \cite{B-three}, we recall that any simple
holomorphic eta quotient of a prime power level is irreducible.
So, in particular, the above theorem implies:
\begin{co}
  For any integer $n>3$, the eta quotient $f_{p,n}$ is 
irreducible.
\end{co}
Also, from Corollary~1 in \cite{B-three}, we recall that given an irreducibile
holomorphic eta quotient $f$ of a prime power level, all the rescalings of $f$
by positive integers are irreducible. Thus we obtain:
\begin{co}[Irreducibility criterion for holomorphic eta quotients]
\ \newline Given a holomorphic eta quotient $g$, if there exist
 $n,d\in\N$ and a prime $p$, such that
 $g$ is the rescaling of $f_{p,n}$ by $d$, then $g$ is irreducible.
 Here $f_{p,n}$ is as defined in $(\ref{fpn})$.
\end{co}

With a good amount of numerical evidence, we conjecture that 
\begin{conj}
 For any integer $n>3$ and for any odd prime $p$, there are no simple holomorphic eta quotients of level $p^{n}$ and of weight greater that of $f_{p,n}$.
\end{conj}

\section{Notations and the basic facts}
\label{26.08.2015.1}
By $\N$ we denote the set of positive integers.
For $N\in\N$, by $\D$ we denote the set of divisors of $N$.
For 
$X\in\ZD$, we define the eta quotient $\e^X$
   by
   \begin{equation}\label{3Jan15.1}
    \e^X:=\displaystyle{\prod_{\substack{\ \\\ \\d\in\D}}\eta_d^{X_d}},
 \end{equation}
where $X_d$ is the 
value of $X$ 
at $d\in\D$ whereas $\e_d$ denotes the rescaling of $\e$ by $d$.
Clearly, the level of $\e^X$ divides $N$. In other words, $\e^X$ transforms like a modular form on $\GN$. 
We define the summatory function $\sigma:\ZD\rightarrow\Z$ by \begin{equation}
\sigma(X):=\sum_{\substack{\ \\\ \\d\in\D}}X_d.
\label{30.08.2015.A}\end{equation}
Since $\e$ is of weight $1/2$,  
the weight of $\e^X$ is  $\sigma(X)/2$ for all $X\in\ZD$.

Recall that an eta quotient $f$ on $\GN$ is holomorphic
if it 
does not have any poles at the cusps of $\GN$. Under the action of $\GN$ on $\PQ$
by M\"obius transformation, for 
$a,b\in\Z$ with $\gcd(a,b)=1$,
we have
\begin{equation}
[a:b]\hspace{.1cm}
{{{{\sim}}}}_{\hspace*{-.05cm}{{{{\GN}}}}}
\hspace{.08cm}[a':\gcd(N,b)]\label{19.05.2015} 
\end{equation}
for some $a'\in\Z$ which is coprime to $\gcd(N,b)$ (see \cite{ds}).
We identify $\PQ$ with $\Q\cup\{\infty\}$ via the canonical bijection that maps $[\alpha:\lambda]$ to 
$\alpha/\lambda$ if $\lambda\neq0$ and to $\infty$ if $\lambda=0$. 
For $s\in\Q\cup\{\infty\}$ and a weakly holomorphic modular form $f$ on $\GN$, the order of $f$ at the cusp $s$ of $\GN$ is 
the exponent of 
 \hyperlink{queue}{$q^{{1}/{w_s}}$} 
occurring with the first nonzero coefficient in the 
$q$-expansion of $f$
at the cusp $s$,
where $w_s$ is the width of the cusp $s$ (see \cite{ds}, \cite{ra}).
The following is a minimal set of representatives of the cusps of $\GN$ (see \cite{ds}, \cite{ymart}):
\begin{equation}
\cl{S}_N:=\Big{\{}\frac{a}{t}\in\Q\hspace{2.5pt}{\di}\hspace{2.5pt}t\in\cl{D}_N,\hspace{2pt} a\in\bb{Z}, 
 \hspace{2pt}\gcd(a,t)=1\Big{\}}/\sim\hspace{1.5pt},
\end{equation}
where $\dfrac{a}{t}\sim\dfrac{b}{t}$ if and only if $a\equiv b\pmod{\gcd(t,N/t)}$.
For $d\in\D$ and for $s=\dfrac{a}t\in\cl{S}_N$ with $\gcd(a,t)=1$, we have
\begin{equation}
 \ord_s(\e_d\hspace{1pt};\GN)= \frac{N\cdot\gcd(d,t)^2}{24\cdot d\cdot\gcd(t^2,N)}\in\frac1{24}\N 
\label{26.04.2015}\end{equation}
(see 
\cite{ymart}). 
It is easy to check the above inclusion  
when $N$ is a prime power. 
The general case 
follows by multiplicativity (see (\ref{27.04.2015}) and (\ref{13May})).
It follows that for all $X\in\ZD$, we have
\begin{equation}
  \ord_s(\e^X\hspace{1pt};\GN)= \frac1{24}\sum_{\substack{\ \\\ \\d\in\D}}\frac{N\cdot\gcd(d,t)^2}{d\cdot\gcd(t^2,N)}X_d\hspace{1.5pt}. 
\label{27.04}\end{equation}
 In particular, 
that implies
\begin{equation}
 \ord_{a/t}(\e^X\hspace{1pt};\GN)=\ord_{1/t}(\e^X\hspace{1pt};\GN)
\label{27.04.2015.1}\end{equation}
for all $t\in\D$ and for all the $\varphi(\gcd(t,N/t))$ inequivalent cusps of $\GN$
represented by rational numbers
of the form $\dfrac{a}{t}\in\cl{S}_N$ with $\gcd(a,t)=1$.
The index of $\GN$ in $\operatorname{SL}_2(\Z)$ is given by
\begin{equation}
\psi(N):= 
N\cdot\hspace*{-.2cm}\prod_{\substack{p|N\\\text{$p$ prime}}}\hspace*{-.2cm}\left(1+\frac{1}{p}\right)
\label{27.04.2015.A}\end{equation}
(see \cite{ds}).
The \emph{valence formula} for $\GN$ (see \cite{ra}) states:
\begin{equation}
\sum_{\substack{\ \\\ \\ P\in\hspace*{1pt}\GN\backslash\H}}\frac1{n_P}\cdot\ord_P(f)
\hspace{2pt}+\sum_{\substack{\ \\\ \\ s\in\hspace{.5pt}\cl{S}_N}}\ord_{s}(f\hspace{1pt};\GN) 
=\frac{k\cdot\psi(N)}{24}\hs,\label{3.2Aug}
\end{equation}
where $k\in\Z$, $f:\H\rightarrow\C$ is a meromorphic function that transforms like a modular forms of weight
$k/2$ on $\GN$ which is also meromorphic at the cusps of $\GN$ and $n_P$ is the number of elements in the 
stabilizer of $P$ in the group $\GN/\{\pm I\}$, where $I\in\operatorname{SL}_2(\Z)$ denotes the identity matrix.
In particular, if $f$ is an eta quotient, then
from (\ref{3.2Aug}) we obtain
\begin{equation}
\sum_{\substack{\ \\\ \\ s\in\hspace{.5pt}\cl{S}_N}}\ord_{s}(f\hspace{1pt};\GN) 
=\frac{k\cdot\psi(N)}{24}\hs,
\label{27.04.2015.2}\end{equation}
because  eta quotients do not have poles or zeros on $\H$. 
it follows from  (\ref{27.04.2015.2})
and from 
 (\ref{27.04.2015.1}) that for an eta quotient $f$ of weight $k/2$ on $\GN$, the valence formula
 further reduces to
\begin{equation}
\sum_{{\substack{\ \\\ \\ t\hspace{.5pt}\idi N}}} 
\varphi(\gcd(t,N/t))\cdot\ord_{1/t}(f\hspace{1pt};\GN) 
=\frac{k\cdot\psi(N)}{24}\hspace{1pt}.
\label{27.04.2015.3}
\end{equation}
Since $\ord_{1/t}(f\hspace{1pt};\GN)\in\frac{1}{24}\Z$ (see (\ref{26.04.2015})), from
(\ref{27.04.2015.3}) it follows that of any particular weight, there are only finitely many holomorphic eta quotients
on $\GN$. More precisely, the number of holomorphic eta quotients of weight $k/2$
on $\GN$ is at most the number of solutions of the following equation
\begin{equation}
 \sum_{\substack{\ \\\ \\t\in\D}}{\varphi(\gcd(t,N/t))}\cdot x_t=k\cdot\psi(N)
\label{27.04.2015.3Z}\end{equation}
in nonnegative integers $x_t$. 

We define the \emph{\hypertarget{om}{order map}}
$\cl{O}_N:\ZD\rightarrow\frac1{24}\ZD$ of level $N$ as the map which sends $X\in\ZD$ to the ordered set of
orders of the eta quotient $\e^X$ at the cusps $\{1/t\}_{\substack{\ \\ t\in\D}}$ of $\GN$. Also, we define the
\emph
{order matrix} $A_N\in\Z^{\D\times\D}$ of level $N$ by 
\begin{equation}
 A_N(t,d):=24\cdot\ord_{1/t}(\e_d\hspace{1pt};\GN) 
\label{27.04.2015}\end{equation}
for all $t,d\in\D$.
 For example, for a prime power $p^n$, 
we have
\begin{equation}
A_{p^n}=\begin{pmatrix}
 \vspace{5.8pt}p^n &p^{n-1} &p^{n-2} 
&\cdots  &p &1\\
\vspace{5.8pt} p^{n-2} &p^{n-1} &p^{n-2} 
&\cdots  &p &1\\
 \vspace{5.8pt}p^{n-4} &p^{n-3} &p^{n-2} 
&\cdots  &p &1\\
\vspace{5.8pt} \vdots &\vdots &\vdots 
&\cdots &\vdots &\vdots\\
\vspace{5.8pt} 1 &p &p^2 
&\cdots &p^{n-1} &p^{n-2}\\
 1 &p &p^2 
&\cdots  &p^{n-1} &p^n
\end{pmatrix}.
 \label{23July}
\end{equation}
By linearity of the order map, we have 
\begin{equation}
\cl{O}_N(X)=\frac1{24}\cdot A_NX\hspace{1.5pt}. 
\label{28.04}\end{equation}
For $r\in\N$, if $Y,Y'\in\Z^{\D^{\hspace{.5pt}r}}$ is 
such that $Y-Y'$ is
nonnegative at each element of $\D^{\hspace{.5pt}r}$, then
we write $Y\geq Y'$. 
In particular, for $X\in\ZD$, the eta quotient $\e^X$ is holomorphic if and only if $A_NX\geq0$.

From $(\ref{27.04.2015})$ and $(\ref{26.04.2015})$, we note that $A_N(t,d)$ is multiplicative in $N, t$ and $d$.
Hence, it follows that
\begin{equation}
 A_N=\bigotimes_{\substack{\ \\\ \\ p^n\|N\\\text{$p$ prime}}}A_{p^n},
\label{13May}\end{equation}
where by \hspace{1.5pt}$\otimes$\hspace{.2pt},  we denote the Kronecker product of matrices.\footnote{Kronecker product of matrices is
not commutative. However, 
since any given ordering of the primes dividing $N$ induces a lexicographic ordering on $\cl{D}_N$ 
with which the entries of $A_N$ are indexed, 
Equation (\ref{13May}) makes sense for all possible 
orderings of the primes dividing $N$.} 

It is easy to verify that for a prime power $p^n$, 
the matrix $A_{p^n}$ is invertible with the tridiagonal inverse: 
\begin{equation}
A_{p^n}^{-1}=
\frac{1}{p^{n}\cdot(p-\frac1p)}
\begin{pmatrix}
\hspace{6pt}p &\hspace*{-6pt}-p &  
& &  & \\
\vspace{5pt}\hspace*{-1pt}-1 & \hspace*{-2pt}p^2+1 &\hspace*{-3pt}-p^2 
& &\textnormal{\Huge 0} & \\
\vspace{5pt}&\hspace*{-6pt}-p & \hspace*{-4pt}p\cdot(p^2+1) &\hspace*{-2pt}-p^3 
 &  &  \\
 &  &\hspace*{-5pt}\ddots 
 &\hspace*{-6pt}\ddots & \ddots & \\
\vspace{5pt}\hspace{2pt} &\textnormal{\Huge 0}  &  
 &\hspace*{-7pt}-p^2 & \hspace*{-4pt}p^2+1 &\hspace*{-4pt}-1\hspace{2pt}\\
\hspace{2pt} &  & 
&  &\hspace*{-6pt}-p & p\hspace{2pt}
\end{pmatrix},
 \label{r1}\end{equation}
where 
for each positive integer $j<n$, 
the nonzero entries of the column  $A_{p^n}^{-1}(\un\hs,p^j)$ are the same as 
those of the column $A_{p^n}^{-1}(\un\hs,p)$ shifted down by $j-1$ entries and
multiplied with $p^{\min\{j-1,n-j-1\}}$. 
More precisely,
\begin{align}
p^{n}\cdot(p-\frac1p)&
\cdot A_{p^n}^{-1}(p^i,p^j)\nonumber=\\&\begin{cases}
                        \ \hs\ p&\text{if $i=j=0$ or $i=j=n$}\\
                        -p^{\min\{j,n-j\}}&\text{if $|i-j|=1$}\\
                        \ \hs\ p^{\min\{j-1,n-j-1\}}\cdot(p^2+1)&\text{if $0<i=j<n$}\\
                        \ \hspace{1.2pt}\ 0&\text{otherwise.}\label{r11}
\end{cases}
\end{align}
For general $N$, the invertibility of the matrix $A_N$ now 
follows by (\ref{13May}).
Hence, any eta quotient on $\GN$ is uniquely determined by its orders at the set of the cusps
$\{1/t\}_{\substack{\ \\ t\in\D}}$ of $\GN$. In particular, for distinct $X,X'\in\ZD$, we have $\e^X\neq\e^{X'}$. 
The last statement 
is also implied by the uniqueness of $q$-series expansion:
Let $\e^{\widehat{X}}$ and $\e^{\widehat{X}'}$
be the \emph{eta products} (i.~e.  $\widehat{X}, \widehat{X}'\geq0$)
obtained by multiplying $\e^X$ and $\e^{X'}$ with a common denominator. The claim follows by induction on the weight of $\e^{\widehat{X}}$
(or equivalently, the weight of $\e^{\widehat{X}'}$)
when we compare
the corresponding 
first two exponents of $q$
occurring in the $q$-series expansions of 
$\e^{\widehat{X}}$ and $\e^{\widehat{X}'}$.

\section{
Proof of Theorem~$\ref{prp1}$}
\label{17.6Sept}

We shall only prove the theorem for the case where $n$ is even. 
The proof for the case where $n$ is odd is quite similar. 
Let $A_N$ be the order matrix of level $N$ (see (\ref{27.04.2015})).
Since all the entries of $A_N^{-1}$
are rational (see (\ref{13May}) and (\ref{r1})), for each $t\in\D$, there exists a smallest positive integer $m_{{t,N}}$ such that 
$m_{{t,N}}\cdot A_N^{-1}(\un\hs,t)$
has integer entries, where $A_N^{-1}(\un\hs,t)$ denotes the 
column of $A_N$ indexed by $t\in\D$.
We define $B_N\in\Z^{\D\times\D}$ by
\begin{equation}
B_N(\un\hs,t):=m_{{t,N}}\cdot A_N^{-1}(\un\hs,t)\hspace{4pt}\text{for all $t\in\D$.}
\label{28.08.2015}\end{equation}
From the multiplicativity of $A_N^{-1}(d,t)$ in $N$, $d$ and $t$  (see (\ref{13May})),
it follows that $B_N(d,t)$ (see (\ref{28.08.2015})) is also multiplicative in $N$, $d$ and $t$. 
That implies:
\begin{equation}
 B_N=\bigotimes_{\substack{\ \\\ \\ \hspace*{3pt}p\in{{\wp}}_{{N}}\\p^n\|N
 }}
 B_{p^n},
\label{28.08.2015.1}\end{equation}
where ${{{\wp}}_{{N}}}$ denotes the set of prime divisors of $N$.
For a prime $p$, 
from (\ref{28.08.2015}) and 
(\ref{r1}), we have 
 \begin{equation}
 B_{p^n}=
\begin{pmatrix}
\hspace{6pt}p &\hspace*{-6pt}-p &  
& &  & \\
\vspace{5pt}\hspace*{-1pt}-1 & \hspace*{-2pt}p^2+1 &\hspace*{-3pt}-p 
& &\textnormal{\Huge 0} & \\
\vspace{5pt}&\hspace*{-6pt}-p & \hspace*{-4pt}p^2+1 &\hspace*{-2pt}-p 
 &  &  \\
 &  &\hspace*{-5pt}\ddots 
 &\hspace*{-6pt}\ddots & \ddots & \\
\vspace{5pt}\hspace{2pt} &\textnormal{\Huge 0}  &  
 &\hspace*{-7pt}-p & \hspace*{-4pt}p^2+1 &\hspace*{-4pt}-1\hspace{2pt}\\
\hspace{2pt} &  & 
&  &\hspace*{-6pt}-p & p\hspace{2pt}
\end{pmatrix}.
\label{equ1}
\end{equation}
From \cite{B-four}, we recall that if $\eta^X$ is an irreducible holomorphic eta quotient on $\Gamma_0(N)$,
then $X\in\cl{Z}_N\cap\bb{Z}^{\D}$, where 
\begin{equation}
\cl{Z}_N=\bigg{\{}\sum_{d|N}C_dv_d\hspace{2pt}\bigg{|}\hspace{2pt}C_d\in[0,1]\text{ for all } d|N\bigg{\}} 
\end{equation}
and for all $d\in \D$, $v_d=m_du_d$, where $u_d$ is the column of $B_N$ indexed by $d$ and $m_d$ is the smallest
positive integer such that $v_d\in\bb{Z}^{\D}$.  Let $v:=\sum_{d|N}v_d$ and let $F_N=\eta^v$.
Again, from \cite{B-four}, we recall 
that 
given a holomorphic eta quotient $g$ on $\Gamma_0(N)$,  the eta quotient 
$F_N/g$ \hspace{2pt}on $\Gamma_0(N)$ is holomorphic if and only if $g$ corresponds to some point in  $\cl{Z}_N\cap\bb{Z}^{\D}$.

In particular, for $N=p^{2m}$, the eta quotient $F_{p^{2m}}=\eta^v$ is given by 
\begin{equation}F_{p^{2m}}(z)=\left\{\begin{array}{ll}
      \eta(pz)^{p^2-1}&\text{if $m=1$,}\\
\\
(\eta(pz)\hspace{1pt}\eta(p^{2m-1}z))^{p(p-1)}\hspace{1pt}\prod_{r=2}^{2m-2}\eta(p^{r}z)^{(p-1)^2}
                      &\text{if $m>1$.}\end{array}\right.\end{equation}
Let $f_{p,2m}$ be the eta quotient defined in (\ref{fpn}). 
Then we have
\begin{equation}\frac{F_{p^{2m}}(z)}{f_{p,2m}(z)}=\left\{\begin{array}{ll}
    \eta(z)^{p-1}\hspace{1pt}\eta(pz)^{p-2}\hspace{1pt}\eta(p^2z)^{p-1}&\text{if $m=1$,}\\
\\
\eta(z)^{p-1}\hspace{1pt}\eta(pz)^{p-2}\hspace{1pt}\eta(p^{2m}z)^{p-1}\hspace{1pt}\prod_{r=1}^{m-1}\frac{\eta(p^{2r-1}z)\hspace{1pt}\eta(p^{2r+1}z)^{p-1}}{\eta(p^{2r}z)}
                      &\text{if $m>1$.}\end{array}\right.\label{k3}\end{equation}
Since $\frac{\eta(z)\eta(p^2z)^{p-1}}{\eta(pz)}$ is 
a holomorphic eta quotient of level $p^2$, 
it follows that 
$\frac{F_{p^{2m}}(z)}{f_{p,2m}(z)}$ is a holomorphic eta quotient of level $p^{2m}$ for all $m\in\bb{N}$.
Let $X\in\bb{Z}^{\D}$ be such that $f_{p,2m}=\eta^X$. From  
(\ref{k3}), we conclude that $X\in\cl{Z}_N$.
In other words, $Y:=\widetilde{A}_NX$ has all its entries in the interval $[0,1]$. 
From (\ref{fpn}), it easily follows that 
$\ord_\infty(f_{2m,p})=1/24$. 
Since $f_{2m,p}$ is invariant under the Fricke involution on $\Gm_0(p^{2m})$, 
we also have $\ord_0(f_{2m,p})=1/24$, since the Fricke involution interchanges the cusps
$0$ and $\infty$ of $\Gm_0(p^{2m})$. Since $0$ and $1$ (resp. $\infty$ and $1/{p^{2m}}$)
represent the same cusp of $\Gm_0(p^{2m})$, from 
(\ref{r1}) we get that
both the first and the last entries of $Y$ are equal to 
$\frac1{p^{2m-1}(p^2-1)}$.


There exists $U_N, V_N\in GL_{\sigma_0(N)}(\bb{Z})$ and  a diagonal matrix $D_N$ such that $D_N=U_N\times B_N\times V_N$. 
We shall see in the next section that if $N=p^n$ for some prime $p$ and some integer $n>2$, then \begin{equation}D_N=\operatorname{diag}(1,1,\hdots,1,p^{n-1},p^{n-1}(p^2-1))\label{k5}\end{equation}                                                                                                                                                                
and the last two columns of $V_N$ are respectively 
\begin{equation}\cl{C}_{n,1}:=\left\{\begin{array}{ll}
(1,0)^t&\text{if $n=1$,}\\
\\
(-1,0,1)^t&\text{if $n=2$,}\\
\\
(1,1,p,p^2,\hdots,p^{n-3},p^{n-2},0)^t&\text{if $n>2$}
                             \end{array}\right.
\end{equation}
and
\begin{equation}
\cl{C}_{n,2}:=\left\{\begin{array}{ll}
(p,1)^t&\text{\quad if $n=1$,}\\
\\
(p^2,1,1)^t&\text{\quad if $n=2$,}\\   
\\
(p^n,p^{n-2},p^{n-3}\hdots,p,1,1)^t&\text{\quad if $n>2$.}      
             \end{array}\right.
\end{equation}  
Next we 
briefly recall 
an useful tool from Linear Algebra:

\begin{minipage}{14cm}
 \ \quad By elementary row and column operations \cite{a}, one can reduce any matrix $B\in\operatorname{GL}_n(\bb{Z})$ to a diagonal matrix $D$. In other words,
there exists $U, V\in\operatorname{GL}_n(\bb{Z})$ and $D=\mathrm{diag}(d_1,d_2,\hdots,d_n)\in\operatorname{GL}_n(\bb{Z})$ such that $D=U\cdot B\cdot V$. 
Since $U,V\in \operatorname{GL}_n(\bb{Z})$,
we have $U^{-1}\cdot\bb{Z}^n=\bb{Z}^n$ and $V\cdot\bb{Z}^n=\bb{Z}^n$. Therefore,
$$\bb{Z}^n/(B\cdot \bb{Z}^n)=U^{-1}\cdot\bb{Z}^n/(B\cdot V\cdot\bb{Z}^n)\simeq\bb{Z}^n/(U\cdot B\cdot V\cdot\bb{Z}^n)=\bb{Z}^n/(D\cdot \bb{Z}^n)=\bigoplus_{i=1}^n\bb{Z}^n/d_i\bb{Z}^n.$$
The above isomorphism maps the element $\underline{\ell}:=(\ell_1\hdots\ell_n)^t$ \hspace{1pt}of\hspace{2pt} $\bigoplus_{i=1}^n\bb{Z}^n/d_i\bb{Z}^n$ to the element $$U^{-1}\cdot\underline{\ell}\pmod{B\cdot\bb{Z}^n}$$ of $\bb{Z}^n/(B\cdot \bb{Z}^n)$.
 Since $B$ is invertible, there is a bijection between $\bb{Z}^n/(B\cdot \bb{Z}^n)$ and \hspace{2pt}$[0,1)^n\cap B^{-1}\cdot\bb{Z}^n$, given by $$X\pmod{B\cdot\bb{Z}^n}\hspace{.1cm}\mapsto\hspace{.1cm} B^{-1}\cdot X\pmod{\bb{Z}^n}.$$
Composing this bijection with the isomorphism above, we get a bijection between \hspace{2pt}$\bigoplus_{i=1}^n\bb{Z}^n/d_i\bb{Z}^n$ and \hspace{2pt}$[0,1)^n\cap B^{-1}\cdot\bb{Z}^n$, given by 
$$\underline{\ell}\hspace{.1cm}\mapsto\hspace{.1cm} B^{-1}\cdot U^{-1}\cdot\underline{\ell}\pmod{\bb{Z}^n}\hspace{.1cm}=\hspace{.1cm}V\cdot D^{-1}\cdot\underline{\ell}\pmod{\bb{Z}^n}.$$
Now multiplication by $B$ maps \hspace{2pt}$[0,1)^n\cap B^{-1}\cdot\bb{Z}^n$ bijectively to $B\cdot[0,1)^n\cap \bb{Z}^n$.
\end{minipage}

\

\

Let $N=p^{2m}$ and suppose,  
$\eta^X=f_{2m,p}$ is reducible. Let $Y=A_NX$. 
Since $\eta^X$ is reducible
there exists $Y',Y''\in\bb{Z}^{\D}\smallsetminus\{0\}$ with $Y'\geq0$ and $Y''\geq0$ such that $Y=Y'+Y''$ and both 
$B_NY'$ and $B_NY''$ have integer entries. Since $B_N$ is an integer matrix
with determinant $d_N:={p^{2m-1}(p^2-1)}$,  we see that $\frac1{d_N}$ is the least possible entry for 
$Y'$ and $Y''$. Since $Y'+Y''=Y$ has $\frac1{d_N}$ as its first entry, either the first entry of $Y'$ or that of
$Y''$ is zero. Similarly, either the last entry of $Y'$ or that of $Y''$ is zero. But it is easy to show\footnote{From the  
congruence relation (\ref{k10})  \hspace{4pt}(resp. replacing $Y'$ with $Y''$ in (\ref{k10})).}
that if both the first and the last entries of $Y'$  (resp. $Y''$) is zero, then $Y'$ (resp. $Y''$) is entirely zero.
So, without loss of generality, we may assume that the first entry of $Y'$ is $\frac1{d_N}$ and the last entry of $Y'$ is 0.
From the previous section and from the entries of the diagonal matrix $D_N$, we know that there exists $\ell_1\in\{0,1,\hdots,p^{2m-1}-1\}$ and 
$\ell_2\in\{0,1,\hdots,p^{2m-1}(p^2-1)-1\}$ such that \begin{equation}
\frac{\ell_1}{p^{2m-1}}\cdot{\cl{C}_{2m,1}}+\frac{\ell_2}{p^{2m-1}(p^2-1)}\cdot{\cl{C}_{2m, 2}}\equiv Y'\pmod{\bb{Z}^n}.\label{k10}\end{equation}

\ \newline\textbf{Case 1.} $(m=1)$


Equating only the first and the last entries from both sides of 
 (\ref{k10}), we obtain
$$\frac{\ell_1}{p}+\frac{p\ell_2}{p^2-1}\equiv\frac{1}{d_N}\pmod{\bb{Z}} \ \ \ \text{ and } \  \ \ \frac{\ell_1}{p}+\frac{\ell_2}{p(p^2-1)}\equiv0\pmod{\bb{Z}},$$
which together implies that $$\frac{\ell_1}{p}\equiv\frac{1}{d_N}\pmod{\bb{Z}}.$$
But this modular equation has no solution in $\ell_1\in\{0,1,\hdots,p-1\}$. Thus we get a contradiction!

\ \newline\textbf{Case 2.} $(m>1)$

Since the last entries of $Y'$ and $\cl{C}_{2m,1}$ are 0, whereas the last entry of $\cl{C}_{2m,2}$ is 1, it follows that $\ell_2=0$. Since the first entry 
of $\cl{C}_{2m, 1}$ is 1, we get $$\frac{\ell_1}{p^{2m-1}}\equiv\frac1{d_N}\pmod{\bb{Z}}$$
as in the previous case. 
Since as before, this has no solution in $\ell_1\in\{0,1,\hdots,p^{2m-1}-1\}$, we get a contradiction.

\vspace{2pt}
 Hence,
$f_{2m,p}=\eta^X$ is irreducible.\qed

\section{The matrix identities}
 \ \quad We continue to prove that the matrix identities $B=UDV$, $UU'=1$ and $VV'=1$ with $B=B_{p^n}$ as defined in (\ref{equ1}) and $D=D_{p^n}$ as defined in (\ref{k5})
 holds if we define $U=U_{p^n}$, $V=V_{p^n}$, $U'=U'_{p^n}$ and $V'=V'_{p^n}$ as follows for $n=1,2,3$ or $n\geq4$:
%
%
%

\vspace*{5pt}
For $n=1$, we define 
$$
U:=\left(\begin{array}{rr}
0 & -1 \\
1 & p
\end{array}\right),\text{  }
V:=\left(\begin{array}{rr}
1 & p \\
0 & 1
\end{array}\right),\text{  }
U':=\left(\begin{array}{rr}
p & 1 \\
-1 & 0
\end{array}\right)\text{ \hspace{2pt}and\hspace{2pt} }
V':=\left(\begin{array}{rr}
1 & -p \\
0 & 1
\end{array}\right).$$

For $n=2$, we define 

$$U:=\begin{pmatrix}
0& 1& 0\\
0& p& 1\\
1& p& 1
  \end{pmatrix},\text{  }
V:=\begin{pmatrix}
0& -1& p^2\\
0& 0& 1\\
-1& 1& 1
  \end{pmatrix},\text{  }
U'=\begin{pmatrix}
0& -1& 1\\
1& 0& 0\\
-p& 1& 0
  \end{pmatrix}\text{ \hspace{2pt}and\hspace{2pt} }$$
$$V'= \begin{pmatrix}
-1& p^2 + 1& -1\\
-1& p^2& 0\\
0& 1& 0
  \end{pmatrix}.
$$

\vspace*{5pt}
For $n=3$, we define

\begin{alignat*}{3}U&:=\begin{pmatrix}
0&-1&-p&-p^2\\
0&0&-1&-p\\
0&0&-p&-(p^2+1)\\
1 &p &p^2 &p^3
  \end{pmatrix},\text{  }
&V&:=\begin{pmatrix}
1&0&1&p^3\\
0&0&1&p\\
0&-1&p&1\\
0&0&0&1
  \end{pmatrix},\\
\\
U'&:=\begin{pmatrix}
p&0&0&1\\
-1&p&0&0\\
0&-(p^2+1)&p&0\\
0&p&-1&0
  \end{pmatrix}\text{ \hspace{2pt}and\hspace{2pt} }
&V'&:= \begin{pmatrix}
1 &-1 &0 &-p(p^2-1)\\
0 &p &-1 &-(p^2-1)\\
0&1&0&-p\\
0&0&0&1
\end{pmatrix}.
\end{alignat*}

For $n>3$:

\vspace{2pt}\quad We define $U=(U_{i,j})_{\hs0\leq i,j\leq n}$ 
by
\begin{equation}\begin{tabular}{ l|c|c| }
\multicolumn{1}{r}{}
 &  \multicolumn{1}{c}{$j=0$}
 & \multicolumn{1}{c}{$j>0$}\\
\cline{2-3}
\vspace*{-.45cm}
&&\\
$i<n-1$ & 0 & $\left\{\begin{array}{cl}
-p^{j-i-1}&\text{if  $j>i$,}\\  
0&\text{otherwise.}    
      \end{array}\right.$\\
\cline{2-3}
\vspace*{-.45cm}
&&\\
$i=n-1$ & 0 & $-p^{n-j}\cdot\frac{p^{2(j-1)}-1}{p^2-1}$ \\
\cline{2-3}
$i=n$ &1 & $p^j$\\
\cline{2-3}
\end{tabular}\hspace{4pt}. 
\label{u}\end{equation}

\vspace{.5cm}\quad We define $V=(V_{i,j})_{\hs0\leq i,j\leq n}$ 
by 
\begin{equation}\begin{tabular}{ l|c|c|c|c| }
\multicolumn{1}{r}{}
 & \multicolumn{1}{c}{$j=0$}
 &  \multicolumn{1}{c}{$0<j<n-1$}
 & \multicolumn{1}{c}{$j=n-1$} 
 & \multicolumn{1}{c}{$j=n$}\\
\cline{2-5}
$i=0$&1 &0&1&$p^n$\\
\cline{2-5}
\vspace*{-.45cm}
&&&&\\
$0<i<n$ & 0& $\left\{\begin{array}{cl}
-p^{i-j-1}&\text{if  $i>j$,}\\  
0&\text{otherwise.}    
      \end{array}\right.$&$p^{i-1}$&$p^{n-i-1}$\\
\cline{2-5}
$i=n$ &0 & $0$&0&1\\
\cline{2-5}
\end{tabular}\hspace{4pt}. 
\label{v}\end{equation}

\quad We define $U'=(U'_{i,j})_{\hs0\leq i,j\leq n}$ by 
\begin{equation}\begin{tabular}{ l|c|c|c|c|c| }
\multicolumn{1}{r}{}
 & \multicolumn{1}{c}{$j=0$}
 &  \multicolumn{1}{c}{$0<j<n-2$}
 & \multicolumn{1}{c}{$j=n-2$} 
 & \multicolumn{1}{c}{$j=n-1$}
 & \multicolumn{1}{c}{$j=n$}
\\
\cline{2-6}
$i=0$&$p$ &0&0&0&1\\
\cline{2-6}
$0<i<n-1$ & $\begin{matrix}-1\hspace{8pt}\\0\\ \vdots\\0\end{matrix}$& $\left\{\begin{array}{cl}
p&\text{if  $i=j$,}\\
-1\hspace{8pt}&\text{if  $i=j+1$,}\\
0&\text{otherwise.}    
      \end{array}\right.$&$\begin{matrix}0\\ \vdots\\0\\p\end{matrix}$&0&0\\
\cline{2-6}
$i=n-1$ & 0 & $-p^{n-j}$&$-(p^2+1)$&$p$&0\\
\cline{2-6}
$i=n$ &0 &$p^{n-j-1}$&$p$&$-1\hspace{8pt}$&0\\
\cline{2-6}
\end{tabular}\hspace{2.5pt}. 
\nonumber \end{equation}
\begin{equation}\label{u'}\end{equation}
\quad We define $V'=(V'_{i,j})_{\hs0\leq i,j\leq n}$
by 
\begin{equation}\begin{tabular}{ l|c|c|c| }
\multicolumn{1}{r}{}
 & \multicolumn{1}{c}{$j=0$}
 &  \multicolumn{1}{c}{$0<j<n$}
 & \multicolumn{1}{c}{$j=n$}\\
\cline{2-4}
$i=0$&1 &$\begin{matrix}-1\quad&0\quad&\cdots\quad&0\quad\end{matrix}$&$-p^{n-2}(p^2-1)$\\
\cline{2-4}
$0<i<n-1$ & 0& $\left\{\begin{array}{cl}
p&\text{if  $i=j$,}\\
-1\hspace{8pt}&\text{if  $i=j-1$,}\\
0&\text{otherwise.}    
      \end{array}\right.$&$-p^{n-i-2}(p^2-1)$\\
\cline{2-4}
$i=n-1$ & 0 & $\begin{matrix}1\quad&0\quad&\cdots\quad&0\quad\end{matrix}$&$-p^{n-2}$\\
\cline{2-4}
$i=n$ &0 & $0$&1\\
\cline{2-4}
\end{tabular}\hspace{4pt}. 
\label{v'}\end{equation}

\

\begin{pr} Given $n\in\bb{N}$, let $U,V,U'$ and $V'$ be the matrices as defined above. For
$N=p^n$, we set the matrices $B=B_N$ and $D=D_N$ as in equations equations (\ref{equ1}) and (\ref{k5}).
Then we have 
$$UU'=I,\hspace{3pt} VV'=I,\hspace{3pt}and\hspace{3pt}D=UBV.$$
\end{pr}

\begin{proof}
If $n\leq 3$, these identities hold trivially. If $n>3$, the proofs of the 
equalities of the corresponding matrix entries 
in each of these matrix relations
involve (at most) summation of some geometric series. For example,
consider the identity $D=UBV$.
 It is equivalent to $DV'=UB$, assuming  $VV'=I$. Now  from (\ref{k5}) and (\ref{v'}),  we see that for $i,j\in\{0,\hdots,n\}$, the $(i,j)$-th entry of $DV'$ is given by
\begin{equation}\begin{tabular}{ l|c|c|c| }
\multicolumn{1}{r}{}
 & \multicolumn{1}{c}{$j=0$}
 &  \multicolumn{1}{c}{$0<j<n$}
 & \multicolumn{1}{c}{$j=n$}\\
\cline{2-4}
$i=0$&1 &$\begin{matrix}-1\quad&0\quad&\hdots\quad&0\quad\end{matrix}$&$-p^{n-2}(p^2-1)$\\
\cline{2-4}
$0<i<n-1$ & 0& $\left\{\begin{array}{cl}
p&\text{if  $i=j$,}\\
-1\hspace{8pt}&\text{if  $i=j-1$,}\\
0&\text{otherwise.}    
      \end{array}\right.$&$-p^{n-i-2}(p^2-1)$\\
\cline{2-4}
$i=n-1$ & 0 & $\begin{matrix}p^{n-1}\quad&0\quad&\hdots\quad&0\quad\end{matrix}$&$-p^{2n-3}$\\
\cline{2-4}
$i=n$ &0 & $0$&$p^{n-1}(p^2-1$)\\
\cline{2-4}
\end{tabular}\hspace{8pt}.\end{equation}
If we consider the case $0<i<n-1$ and $0<j<n$, then from
$(\ref{u})$ and $(\ref{equ1})$ it follows that the product of the $i$-th row of $U$
and the $j$-th column of $B$ is \begin{align*}-\sum_{k=i+1}^np^{k-i-1}\hspace{2pt}B_{k,\hspace{1pt}j}&=\sum_{k=i+1}^np^{k-i-1}(p\hspace{1pt}\delta_{|k-j|\hspace{1pt},\hspace{1pt}1}-(p^2+1)\hspace{1pt}\delta_{k,\hspace{1pt}j})\\
&=\sum_{k=\max\{i+1,\hspace{1pt}j-1\}}^{j+1}p^{k-i-1}(p\hspace{1pt}\delta_{|k-j|\hspace{1pt},\hspace{1pt}1}-(p^2+1)\hspace{1pt}\delta_{k,\hspace{1pt}j})\\
&=\left\{\begin{array}{ll}
\hspace{8pt}0&\text{ if $j<i$,}\\
\hspace{8pt}p&\text{ if $j=i$,}\\
-1&\text{ if $j=i+1$,}\\
\hspace{8pt}0&\text{ if $j>i+1$,}
\end{array}\right.
\end{align*}
where $\delta$ is the usual Kronecker delta function. So, the claim holds in this case.

Again, if we consider the case $i=n-1$ and $0<j<n$, then the product of the $i$-th row of $U$
and the $j$-th column of $B$ is \begin{align*}-\sum_{k=1}^np^{n-k}\left(p^{2(k-1)}-1\right)\hspace{2pt}B_{k,\hspace{1pt}j}&=\sum_{k=2}^np^{n-k}\left(p^{2(k-1)}-1\right)\left(p\hspace{1pt}\delta_{|k-j|\hspace{1pt},\hspace{1pt}1}-(p^2+1)\hspace{1pt}\delta_{k,\hspace{1pt}j}\right)\\
&=\sum_{k=\max\{2,\hspace{1pt}j-1\}}^{j+1}p^{n-k}\left(p^{2(k-1)}-1\right)\left(p\hspace{1pt}\delta_{|k-j|\hspace{1pt},\hspace{1pt}1}-(p^2+1)\hspace{1pt}\delta_{k,\hspace{1pt}j}\right)\\
&=\left\{\begin{array}{ll}
p^{n-1}&\text{ if $j=1$,}\\
0&\text{ otherwise,}
\end{array}\right.
\end{align*}
where the first equality holds since $p^{2(k-1)}=1$ for $k=1$. Thus, the claim also holds in this case.
The rest of the proof is quite similar and only requires some more straightforward checks as above.
\end{proof}
\bibliography{infinite-bibtex}

\begin{thebibliography}{10}
\providecommand{\url}[1]{#1}
\csname url@samestyle\endcsname
\providecommand{\newblock}{\relax}
\providecommand{\bibinfo}[2]{#2}
\providecommand{\BIBentrySTDinterwordspacing}{\spaceskip=0pt\relax}
\providecommand{\BIBentryALTinterwordstretchfactor}{4}
\providecommand{\BIBentryALTinterwordspacing}{\spaceskip=\fontdimen2\font plus
\BIBentryALTinterwordstretchfactor\fontdimen3\font minus
  \fontdimen4\font\relax}
\providecommand{\BIBforeignlanguage}[2]{{%
\expandafter\ifx\csname l@#1\endcsname\relax
\typeout{** WARNING: IEEEtranS.bst: No hyphenation pattern has been}%
\typeout{** loaded for the language `#1'. Using the pattern for}%
\typeout{** the default language instead.}%
\else
\language=\csname l@#1\endcsname
\fi
#2}}
\providecommand{\BIBdecl}{\relax}
\BIBdecl

\bibitem{ak}
G.~E. Andrews and K.~Eriksson, \emph{Integer partitions}.\hskip 1em plus 0.5em
  minus 0.4em\relax \hspace*{-.65cm}Cambridge University Press, Cambridge,
  2004, {MR~2122332}, {Zbl~1073.11063},
  \url{http://dx.doi.org/10.1017/CBO9781139167239}.

\bibitem{B-three}
S.~Bhattacharya, ``Determination of irreducibility of a holomorphic eta
  quotient,'' preprint, \url{http://arxiv.org/pdf/1602.03087.pdf}.

\bibitem{B-four}
------, ``Finiteness of irreducible holomorphic eta quotients of a given
  level,'' preprint, \url{http://arxiv.org/pdf/1602.02814.pdf}.

\bibitem{B-two}
------, ``{Factorization of holomorphic eta quotients},'' Ph.D thesis,
  Rheinische Friedrich-Wilhelms-Universit\"at Bonn, 2014,
  \url{hss.ulb.uni-bonn.de/2014/3711/3711.pdf}.

\bibitem{bswz}
J.~M. {Borwein}, A.~{Straub}, J.~{Wan}, and W.~{Zudilin}, ``{Densities of short
  uniform random walks.}'' \emph{{Can. J. Math.}}, vol.~64, no.~5, pp.
  961--990, 2012, {Zbl~1296.33011},
  \url{http://dx.doi.org/10.4153/CJM-2011-079-2}.

\bibitem{dc}
D.~{Choi}, ``{Spaces of modular forms generated by eta-quotients},''
  \emph{{Ramanujan J.}}, vol.~14, no.~1, pp. 69--77, 2007, {MR~2298641},
  {Zbl~1197.11052}, \url{http://dx.doi.org/10.1007/s11139-006-9007-3}.

\bibitem{cs}
S.~Chowla and A.~Selberg, ``On {E}pstein's zeta-function,'' \emph{J. Reine
  Angew. Math.}, vol. 227, pp. 86--110, 1967, {MR~0215797}.

\bibitem{coh}
H.~Cohen, \emph{Number Theory, Volume II: Analytic and Modern Tools}.\hskip 1em
  plus 0.5em minus 0.4em\relax Springer-Verlag, New York, 2007, {Graduate Texts
  in Mathematics}. 240, \url{http://dx.doi.org/10.1007/978-0-387-49894-2}.

\bibitem{rd}
R.~Dedekind, ``Schreiben an {H}errn {B}orchardt \"uber die {T}heorie der
  elliptischen {M}odul-{F}unctionen,'' \emph{J. Reine Angew. Math.}, vol.~83,
  pp. 265--292, 1877, {MR~1579737},
  \url{http://dx.doi.org/10.1515/crll.1877.83.265}.

\bibitem{ds}
F.~{Diamond} and J.~{Shurman}, \emph{A First Course in Modular Forms}.\hskip
  1em plus 0.5em minus 0.4em\relax Springer-Verlag, New York, 2005, {Graduate
  Texts in Mathematics. 228}, \url{http://dx.doi.org/10.1007/b138781}.

\bibitem{drin}
V.~G. Drinfel'd, ``Two theorems on modular curves,'' \emph{Funkcional. Anal. i
  Prilo\v zen.}, vol.~7, no.~2, p. 83—84, 1973, {MR~0318157}.

\bibitem{h}
K.~Harada, \emph{``{M}oonshine'' of finite groups}, ser. EMS Series of Lectures
  in Mathematics.\hskip 1em plus 0.5em minus 0.4em\relax European Mathematical
  Society (EMS), Z\"urich, 2010, {MR~2722318},
  \url{http://dx.doi.org/10.4171/090}.

\bibitem{a}
R.~A. Horn and C.~R. Johnson, \emph{Topics in {M}atrix {A}nalysis}.\hskip 1em
  plus 0.5em minus 0.4em\relax Cambridge University Press, 1994, {MR~1288752}.

\bibitem{k}
V.~G. Kac, ``{I}nfinite-dimensional algebras, {D}edekind's {$\eta $}-function,
  classical {M}\"obius function and the very strange formula,'' \emph{Adv. in
  Math.}, vol.~30, no.~2, pp. 85--136, 1978, {MR~563927},
  \url{http://dx.doi.org/10.1016/0001-8708(78)90033-6}.

\bibitem{kil}
L.~J.~P. Kilford, ``Generating spaces of modular forms with
  {$\eta$}-quotients,'' \emph{JP J. Algebra Number Theory Appl.}, vol.~8,
  no.~2, pp. 213--226, 2007, {MR~2406859}.

\bibitem{kz}
P.~Kleban and D.~Zagier, ``Crossing probabilities and modular forms,'' \emph{J.
  Statist. Phys.}, vol. 113, no. 3-4, pp. 431--454, 2003, {MR~2013692},
  \url{http://dx.doi.org/10.1023/A:1026012600583}.

\bibitem{b}
G.~K{\"o}hler, \emph{Eta products and theta series identities}, ser. Springer
  Monographs in Mathematics.\hskip 1em plus 0.5em minus 0.4em\relax Springer,
  Heidelberg, 2011, {MR~2766155},
  \url{http://dx.doi.org/10.1007/978-3-642-16152-0}.

\bibitem{km}
W.~Kohnen and G.~Mason, ``On generalized modular forms and their
  applications,'' \emph{Nagoya Math. J.}, vol. 192, pp. 119--136, 2008,
  {MR~2477614}.

\bibitem{lo}
R.~J. Lemke~Oliver, ``Eta-quotients and theta functions,'' \emph{Adv. Math.},
  vol. 241, pp. 1--17, 2013, {MR~3053701}, {Zbl~1282.11030},
  \url{http://dx.doi.org/10.1016/j.aim.2013.03.019}.

\bibitem{ymart}
Y.~Martin, ``Multiplicative {$\eta$}-quotients,'' \emph{Trans. Amer. Math.
  Soc.}, vol. 348, no.~12, pp. 4825--4856, 1996, {MR~1376550},
  \url{http://dx.doi.org/10.1090/S0002-9947-96-01743-6}.

\bibitem{c}
G.~Mersmann, ``Holomorphe $\eta$-produkte und nichtverschwindende ganze
  modul-formen f\"ur {$\Gamma_0(N)$},'' Diplomarbeit, Rheinische
  Friedrich-Wilhelms-Universit\"at Bonn, 1991,
  \url{https://sites.google.com/site/soumyabhattacharya/miscellany/Mersmann.pdf}.

\bibitem{r}
H.~Rademacher, \emph{Topics in analytic number theory}.\hskip 1em plus 0.5em
  minus 0.4em\relax Springer-Verlag, New York-Heidelberg, 1973, {MR~0364103}.

\bibitem{ra}
R.~A. Rankin, \emph{Modular forms and functions}.\hskip 1em plus 0.5em minus
  0.4em\relax Cambridge University Press, Cambridge, 1977, {MR~0498390}.

\bibitem{pw}
A.~van~der Poorten and K.~S. Williams, ``Values of the {D}edekind eta function
  at quadratic irrationalities,'' \emph{Canad. J. Math.}, vol.~51, no.~1, pp.
  176--224, 1999, {MR~1692895}, {Zbl~0936.11026},
  \url{http://dx.doi.org/10.4153/CJM-1999-011-1}.

\end{thebibliography}
\bibliographystyle{IEEEtranS}
\nocite{*}
\end{document}